\documentclass[a4paper,10pt]{article}
\usepackage[latin1]{inputenc}
\usepackage{amsmath}
\usepackage{amsfonts}
\usepackage{amsthm}
\usepackage[dvipsnames]{xcolor} 
\usepackage{graphicx}

\newtheorem{theorem}{Theorem}
\newtheorem{lemma}{Lemma}
\newtheorem{proposition}{Proposition}
\newtheorem*{theorem*}{Theorem}

\makeatletter
\newcommand{\vast}{\bBigg@{4}}
\newcommand{\Vast}{\bBigg@{5}}
\makeatother

\DeclareMathOperator{\dist}{dist}

\title{Lattice points in bodies of revolution II}
\author{Fernando Chamizo \and Carlos Pastor}
\date{}

\begin{document}

\maketitle

\begin{abstract}
In \cite{chamizo} it was shown that when a three-dimensional smooth convex body has rotational symmetry around a coordinate axis one can find better bounds for the lattice point discrepancy than what is known for more general convex bodies. To accomplish this, however, it was necessary to assume a non-vanishing condition on the third derivative of the generatrix. In this article we drop this condition, showing that the aforementioned bound holds for a wider family of revolution bodies, which includes those with analytic boundary.
A novelty in our approach is that, besides the usual analytic methods, it requires studying some Diophantine properties of the Taylor coefficients of the phase on the Fourier transform side.
\end{abstract}

\section{Introduction}

Let $\mathbb{E} \subset \mathbb{R}^d$, $d > 1$, be a compact convex body with non-empty interior, and whose boundary $\partial \mathbb{E}$ is a smooth $(d-1)$-submanifold with positive Gaussian curvature (for short, a \emph{smooth convex body}). Denote by $\mathcal{N}(R)$ the number of points of $\mathbb{Z}^d$ lying in $\mathbb{E}$ after being dilated by a factor $R > 1$, \emph{i.e.},
\[
  \mathcal{N}(R) = \#
  \{ 
    \vec{n} \in \mathbb{Z}^d \, : \,
    \vec{n}/R \in \mathbb{E}
  \}.
\]
A central question in lattice point theory consists in estimating how big the discrepancy
\[
  \mathcal{E}(R) = \mathcal{N}(R) - R^d|\mathbb{E}|
\]
can be for a particular choice of $\mathbb{E}$, where $|\mathbb{E}|$ stands for the volume of $\mathbb{E}$. We denote by $\alpha_\mathbb{E}$ the minimal exponent (using Landau's notation)
\[
  \alpha_\mathbb{E} = \inf 
  \big\{
    \alpha \, : \, \mathcal{N}(R) = R^d |\mathbb{E}| + O(R^\alpha)
  \big\}.
\]
A geometrical observation originally due to Gauss for the circle \cite{gauss}, shows that $\alpha_\mathbb{E} \leq d - 1$. This elementary upper bound has been improved by numerous authors. The state of the art is the following: Huxley \cite{huxley} has proved $\alpha_\mathbb{E} \leq 131/208$ for $d = 2$, and Guo \cite{guo} $\alpha_\mathbb{E} \leq 73/158$ for $d = 3$, and $\alpha_\mathbb{E} \leq (d^2 + 3d + 8)/(d^3 + d^2 + 5d + 4)$ for $d \geq 4$.

For some particular choices of $\mathbb{E}$ it is possible to do better. For example, when $\mathbb{E}$ is the $d$-dimensional unit ball, the problem is essentially solved for $d \geq 4$: the equality $\alpha_\mathbb{E} = d - 2$ follows from classical results on representations of integers by quadratic forms. The best known upper bound for the three-dimensional unit ball is $\alpha_\mathbb{E} \leq 21/16$ shown by Heath-Brown \cite{heathbrown}, while for the circle a recent preprint of Bourgain and Watt \cite{bourgainwatt} improves Huxley's result to $\alpha_\mathbb{E} \leq 517/824$.

The aim of the article \cite{chamizo} was to show that, albeit the substantial differences in lattice point problems for general smooth convex bodies and balls that the previous results evidentiate, if one assumes rotational symmetry around a coordinate axis then one can obtain intermediate results even from the simplest van der Corput's estimate. 
We consider three-dimensional smooth convex bodies of the form
\begin{equation}\label{bor}
\mathbb{E}
=
\big\{(x,y,z)\in\mathbb{R}^3\;:\;
f_2(r)\le z\le f_1(r),\ 0\le r\le r_\infty
\big\}
\end{equation}
where $r=\sqrt{x^2+y^2}$.

In other words, $\mathbb{E}$ is the solid generated by the rotation around the $z$-axis of the curve
\[
 \gamma(t)=
 \begin{cases}
  \big(t,0,f_1(t)\big)& 0\le t\le r_\infty
  \\[9pt]
  \big(2r_\infty-t,0,f_2(2r_\infty-t)\big)& r_\infty\le t\le 2r_\infty
 \end{cases}
 \qquad
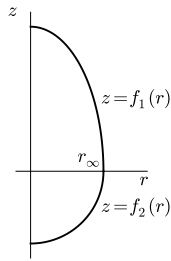
 \begin{tabular}{c}
\includegraphics[height=100pt, keepaspectratio=true]{imag1.eps}
 \end{tabular}
\]

Theorem~1.1 of \cite{chamizo} reads:

\begin{theorem*}
Let $\mathbb{E}\subset\mathbb{R}^3$ be a body of revolution as before and suppose that the functions $\frac 1r f_i'''(r)$ (extended by continuity to $r = 0$) do not vanish for $0 \leq r < r_\infty$, where $i = 1, 2$. Then the inequality $\alpha_\mathbb{E} \leq 11/8$ holds.
\end{theorem*}

This result is not entirely satisfactory, as the extra hypothesis concerning the non-vanishing of the function $\frac 1r f_i'''(r)$ lacks geometrical meaning. This kind of technical hypothesis, which often appear when applying van der Corput's estimates, are usually very difficult to deal with. The problem worsened when multiple exponent pairs were introduced. For instance, until the arrival of the discrete Hardy-Littlewood method \cite{huxley_book}, stating and checking non-vanishing conditions was a substatial task even for the circle and divisor problems (see \cite{hua} and \cite{kolesnik}). The problem still persists when $d>2$. In this article, however, we prove

\begin{theorem}\label{thethm}
Let $\mathbb{E}\subset\mathbb{R}^3$ be a body of revolution as before and suppose that the functions $f_i$ are real analytic for $0 \leq r < r_{\infty}$ and $i = 1, 2$. Then the inequality $\alpha_\mathbb{E} \leq 11/8$ holds. 
\end{theorem}

In fact much less is needed as evident from the proof below. Theorem~\ref{thethm} holds as long as every zero of the function $f_i'''$ is of finite order for both $i = 1, 2$, \emph{i.e.} $f_i'''(r) = 0$ implies we can find an integer $n > 3$ such that $f_i^{(n)}(r) \neq 0$.
As remarked in the next section, the result also holds if in the definition \eqref{bor} we take $r=\sqrt{Q(x+\alpha,y+\beta)}$ with $Q$ a positive definite rational quadratic form and $\alpha,\beta\in\mathbb{R}$. In other words, Theorem~\ref{thethm} extends to the case in which the horizontal sections are rational ellipses with a common center when projected onto the $xy$-plane.

The idea of the proof is the following: we transform the problem via Poisson summation into estimating an exponential sum, as it is customary; and then slice the sum diadically in pieces corresponding to the zeros of $f_i'''(r)$. For the pieces where van der Corput's estimation falls short the phase is almost linear, and we are in position to apply the first derivative test (Kuzmin-Landau inequality \cite{grahamkolesnik}). This, by itself, is not good enough, as the derivative of the phase function might happen to be close to an integer way too often. Showing that this cannot be the case requires --in some ranges-- studying certain diophantine properties of the Taylor coefficient in question. This goes beyond the utterly analytic treatment  in the classical (van der Corput's) theory of exponential sums and vaguely resembles to the situation in \cite{bombieriiwaniec} (the seminal paper for the discrete Hardy-Littlewood method) in which the arithmetic properties of the Taylor coefficients play a fundamental role.

It is worth mentioning that while looking for examples where the non-vanishing condition is blatantly violated, the authors were led to the case of revolution paraboloids (and more generally elliptic paraboloids) for which it turns out that one can explicitely determine $\alpha_\mathbb{E}$ by relating the associated exponential sum to some analytic functions enjoying automorphic properties \cite{chamizopastor}. In some sense a related phenomenon is happening here, as very close to a zero of $\frac 1r f_i'''(r)$ the function $f_i(r)$ essentially looks like a parabola, and some of the arithmetic leaks in in the form of the aforementioned diophantine properties of the Taylor coefficient.

Throughout this article we use Vinogradov's notation $f \ll g$ with the same meaning as Landau's $f = O(g)$, \emph{i.e.}, $|f| \leq C|g|$ for an unspecified constant $C$. We employ $f \gg g$ to denote $g = O(f)$, and $f \asymp g$ when both of these conditions hold. We also abbreviate $e^{2\pi i t}$ by $e(t)$.

\section{The exponential sum}

Our starting point will be the truncated Hardy-Vorono\"i formula given by Proposition~2.1 of \cite{chamizo}. We restate it here for convenience. Fix a smooth even function $\eta \in C^\infty_0\big((-1, 1)\big)$ with $\eta(0) = 1$ and satisfying that the Fourier transform of $\eta(\|\vec{x}\|)$ is positive (this latter condition can be easily fulfilled by considering the convolution of a radial function with itself).

\begin{proposition}\label{pr:hv}
Let $\mathbb{E} \subset \mathbb{R}^3$ be a smooth convex body, $\eta$ as before, and fix $\epsilon > 0$ and $0 < c < 1$. Then for any given $R > 2$ there exists $R' \in \big(R-2, R+2\big)$ such that
\[
  \mathcal{E}(R) = 
  - \frac{R'}{\pi}
  \sum_{\vec{n} \in \mathbb{Z}^3 - \{0\}}
    \eta\big( \delta \|\vec{n}\| \big)
    \frac{ 
      \cos\big( 2 \pi R' g(\vec{n}) \big) 
    }{
      \|\vec{n}\|^2 \sqrt{\kappa(\vec{n})}
    }
  + O\big(R^{2 + \epsilon}\delta\big),
\]
where $\delta = R^{-c}$, $g$ is the \emph{support function} $g(\vec{n}) = \sup \{\vec{x} \cdot \vec{n} \, : \, \vec{x} \in \mathbb{E} \}$ and $\kappa(\vec{n})$ stands for the Gaussian curvature of $\partial \mathbb{E}$ at the point whose unit outer normal is $\vec{n}/\|\vec{n}\|$.
\end{proposition}

To obtain an error term of the form $O\big(R^{11/8 + \epsilon}\big)$ we must choose $c \geq 5/8$, and the larger we pick $c$ the longer the exponential sum we have to bound becomes. The natural choice is therefore $c = 5/8$ and $\delta = R^{-5/8}$.

All the functions of $\vec{n}$ involved in the expression for $\mathcal{E}(R)$ given by Proposition~\ref{pr:hv} are invariant under rotations on the first two variables. Grouping the corresponding terms and applying summation by parts,
\begin{equation}\label{eq:parts}
  \mathcal{E}(R) \ll
  \sup_{ N, M^2 \leq \delta^{-2}}   
    \frac{R^{1+\epsilon}}{N + M^2} 
    \left| 
      \sum_{1 \leq n \leq N} 
        \sum_{1 \leq |m| \leq M} 
          r_2(n) e\big(R'h(n, m)\big)
    \right|    
  + R^{11/8 + \epsilon},
\end{equation}
where $h(n, m) = g\big(\sqrt{n}, 0, m\big)$ and $r_2(n)$ stands for the number of representations of $n$ as sum of two squares (the contribution corresponding to the terms with $n = 0$ or $m = 0$ is negligible). The summation by parts step is justified as long as we can guarantee
\[\int_1^{\delta^{-2}}
 \int_1^{\delta^{-1}} 
 (u + v^2) 
 \left| 
   \frac{\partial^2}{\partial u \partial v} 
   \frac{\eta\big(\delta \sqrt{u+v^2}\big)}
   {(u+v^2) \sqrt{\kappa(\sqrt{u}, 0, v) } }
 \right| du dv
 \ll R^\epsilon,
\]
and a similar bound with $v$ replaced with $-v$. This becomes evident after performing the change of variables $u \mapsto u^2$ and changing to polar coordinates (note $\kappa(u, 0, v)$ depends smoothly on the angle $\theta = \arctan v/u$).

When $r$ in \eqref{bor} is replaced by $r=\sqrt{Q(x+\alpha,y+\beta)}$ with $Q$ a positive definite rational quadratic form, as mentioned in the introduction, rescaling $f_i$ we can assume that the \emph{dual form $Q^*$} (the one having the inverse matrix) is an integral quadratic form. By the properties of the Fourier transform, \eqref{eq:parts} still holds but replacing $r_2(n)$ by
\[
 r_2^*(n)=\frac{1}{\det Q}
 \sum_{Q^*(x,y)=n}e(\alpha x+\beta y)
 \qquad\text{where $x,y\in\mathbb{Z}$}.
\]
See \S6 of \cite{chamizo} for more details. In what follows the only fact that will be used about the arithmetic function $r_2$ is the bound $r_2(n) \ll n^\epsilon$ for any $\epsilon > 0$, also satisfied by $r_2^*(n)$. Therefore all the forthcoming arguments may be readily applied in the context of rational elliptic sections.

If instead of the convex body $\mathbb{E}$ we consider its specular reflection over the plane $z = 0$, we arrive at exactly the same expression \eqref{eq:parts} but with the sign of $m$ reversed (because the invariance of the Fourier transform by symmetries). This means that we can restrict our attention to the half of the sum consisting of those terms with $m > 0$, and then apply the same argument to the specular reflection of $\mathbb{E}$ to bound the other half in the same way. In what follows, therefore, we restrict $m$ to be positive, and rename $f_1$ to $f$ to avoid the excessive use of subindices.

Let $0 = r_0 < r_1 < \cdots < r_{j_0-1}$ be the zeros of $f'''$ in $[0,r_\infty)$ and fix any $r_{j_0}$ satisfying $r_{j_0-1}<r_{j_0}<r_\infty$. Denote $u_j=(f'(r_j))^2$ for $0\le j\le j_0$. We are going to split the summation domain of \eqref{eq:parts} dyadically in $m$ as $m \rightarrow +\infty$, and in $n/m^2$ as it approaches either some $u_{j}$ or $+\infty$ (see figure below); decomposing the sum in an at most a constant times $\log{R}$ number of pieces of the form
\begin{equation}\label{eq:S}
  S(U_1, U_2, M) =
  \mathop{\sum \sum}_{\substack{
    U_1 \leq n/m^2 < U_2 \\
    M \leq m < 2M 
  }     }
    r_2(n) e\big(Rh(n, m)\big).
\end{equation}

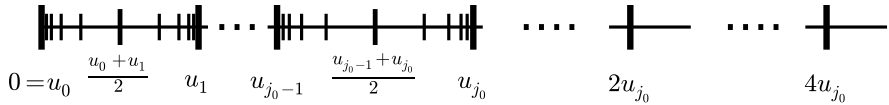
\begin{center}
\includegraphics[height=46pt, keepaspectratio=true]{imag2.eps}
\end{center}

It is clear that after this dyadic subdivision, we can deduce Theorem~\ref{thethm} from the following theorems. The part  $0\le n/m^2<u_{j_0}$ of the double sum in \eqref{eq:parts} is covered by Theorem~\ref{thm:s1}, except for the terms with $n/m^2=u_j$ that can be estimated trivially, and the part $n/m^2\ge u_{j_0}$ is covered by Theorem~\ref{thm:s2}.

\begin{theorem}\label{thm:s1}
Given $\epsilon > 0$ and $0\le j<j_0$, for any $R > 1$, $2M \leq R^{5/8}$
and $0<U\le (u_{j+1}-u_j)/4$ we have
\[
 \big|S(u_{j+1} - 2U, u_{j+1} - U, M)\big|
 +
 \big|S(u_{j} + U, u_{j} + 2U, M)\big|
 \ll M^2R^{3/8+\epsilon}.
\]
\end{theorem}

\begin{theorem}\label{thm:s2}
Given $\epsilon > 0$, for any $R > 1$, $2M \leq R^{5/8}$
and $u_{j_0}\le U\le R^{5/4}M^{-2}$ we have
\[
  S(U, 2U, M) \ll UM^2 R^{3/8 + \epsilon}.
\]
\end{theorem}

\section{Weyl step}

In order to be able to estimate the sum $S$ given by \eqref{eq:S} using the van der Corput method we must first get rid of the arithmetic function $r_2$. We do this by performing a so-called \emph{Weyl step} (\emph{cf.} \S8.2 of \cite{iwanieckowalski}).

\begin{proposition}\label{prop:Weyl}
Let $S$ as before and fix $\epsilon > 0$. For any $1 \leq M \leq R^{5/8}$, $0 < U_1 < U_2 \leq R^{5/4}$ and $1 \leq L \leq M$, satisfying 
$U_2 - U_1 = U$ and $U_2L + 1 \ll UM$, we have
\[
  \big|S(U_1, U_2, M)\big|^2 \ll
  R^\epsilon ({U^2M^6}L^{-1}
  + U M^3  T)
\]
where $T = T(U_1, U_2, M, L)$ is given by
\begin{equation}\label{eq:T}
  T = \frac 1L
  \sum_{1 \leq \ell \leq L}
  \vast|
    \mathop{\sum \sum}_{\substack{
     U_1 \leq n/(m+\ell)^2, n/m^2 < U_2 \\
      M \leq m, m + \ell < 2M 
    }}
    e\big(R(h(n, m + \ell) - h(n, m))\big)
  \vast|.
\end{equation}
\end{proposition}

\begin{proof}
Consider
\[
  \psi_{n, m} = 
  \begin{cases}
    e\big(Rh(n, m)\big) & \text{if } U_1 \leq n/m^2 < U_2 \text{ and } M \leq m < 2M, \\
    0 & \text{otherwise}.
  \end{cases}
\]
It suffices to prove the inequality when $L$ is an integer. We may therefore write
\[
  LS = 
  \sum_{M-L \leq m < 2M} 
  \sum_{U_1 m^2 \leq n < U_2(m + L)^2}
  r_2(n)
  \sum_{1 \leq \ell \leq L} \psi_{n, m + \ell}.
\]
The length of the first sum is $\ll M$ and the length of the second one $\ll UM^2$, hence squaring and applying Cauchy-Schwarz,
\[
  L^2 S^2 \ll
  R^{\epsilon} UM^{3}
  \sum_{M-L \leq m < 2M} 
  \sum_{U_1 m^2 \leq n < U_2(m + L)^2}
  \sum_{1 \leq \ell_1, \ell_2 \leq L}
    \psi_{n, m + \ell_1} \overline{\psi_{n, m + \ell_2}}.
\]
Separating the diagonal contribution $\ell_1 = \ell_2$ and interchanging the summation order, which can be done because $\psi_{n, m}$ keeps track of the summation domain,
\[
  L^2 S^2 \ll
  R^\epsilon LU^2M^6 
  + R^\epsilon U M^3 \Re
    \sum_{1 \leq \ell_2 < \ell_1 \leq L}
    \sum_n
    \sum_m
      \psi_{n, m + \ell_1} \overline{\psi_{n, m + \ell_2}}.
\]
To obtain the desired inequality perform the change of variables $m \mapsto m - \ell_2$ and group the terms corresponding to each value of $\ell = \ell_1 - \ell_2$.
\end{proof}

\section{The function $h$}\label{s:h}

In this section we prove the estimates we need about the function $h$. Note that the convexity of $-f$ implies that $-f': [0, r_{\infty}) \rightarrow \mathbb{R}^+$ is one-to-one, and therefore its inverse function $\phi$ is well-defined. By Lemma~4.1 of \cite{chamizo} we know that
\begin{equation}\label{F}
  \frac{\partial}{\partial m} h(n, m) = F\big(n/m^2\big) \quad \text{where} \quad F(u) = f\big(\phi(\sqrt{u})\big).
\end{equation}

The estimates for $h$ near the ``bad'' points $u_{j}$ will depend on the order of vanishing of $f'''(r)$. By definition, each $u_j$ is the preimage by the function $\phi(\sqrt{u})$ of a zero $r_j$ of $f'''(r)$, except the last one which is added for convenience. If $r_{j} \neq 0$ we define $d_{j}$ as the unique nonnegative integer satisfying $f'''(r) \asymp (r-r_{j})^{d_{j}}$ as $r \rightarrow r_{j}$. For $r_{0} = 0$ we define $d_{0}$ as the unique nonnegative integer satisfying $f'''(r) \asymp r^{2d_{0} + 1}$ as $r \rightarrow 0^+$. We also set $d_{\infty} = -5/2$. 

\begin{lemma}\label{lem:F2}
We have $F'(u) \asymp (1+u)^{-3/2}$ for $0 \leq u < \infty$. We also have $F''(u) \neq 0$ for $u\ne u_j$, $0\le j \leq j_0$, and 
\[
F''(u) \asymp (u-u_j)^{d_j} \quad \text{as} \quad u \rightarrow u_j
\qquad\text{and}\qquad
F''(u) \asymp u^{d_\infty} \quad \text{as} \quad u \rightarrow \infty.
\]
\end{lemma}

\begin{proof}
Let $k(r)$ denote the curvature of $r \mapsto \big(r, f(r)\big)$, which admits the explicit formula (see p. 11 of \cite{spivak})
\begin{equation}\label{eq:k}
  f''(r) = 
  k(r)\big(1+|f'(r)|^2\big)^{3/2},
\end{equation}
and set $c(u) = k\big(\phi(\sqrt{u})\big)$. Differentiating $F$ and recalling that $\phi$ is the inverse function of $-f'$ we obtain (\emph{cf.} the proof of Lemma~4.2 of \cite{chamizo})
\begin{align*}
  F'(u) &= 
    \frac{1}{2c(u)(1+u)^{3/2}}, \\
  F''(u) &= 
    \frac{
      f'''\big(\phi(\sqrt{u})\big)
    }{
      4 \big(c(u)\big)^3(1+u)^{9/2}u^{1/2}
    }.  
\end{align*}

Now all but the last claim of the lemma is clear as $c_1<c(u)<c_2$ for some constants $c_1,c_2>0$ and $\phi(\sqrt{u})$ is a regular function for $u > 0$, and behaves like $C\sqrt{u}$ for some $C \neq 0$ as $u \rightarrow 0^+$. To establish the last claim, we note that by \eqref{eq:k} and L'Hôpital's rule,
\[
  k(r_{\infty}) 
  = \lim_{r \rightarrow r_{\infty}^-} \frac{f''(r)}{\big(f'(r)\big)^3}
  = \lim_{r \rightarrow r_{\infty}^-} \frac{f'''(r)}{3\big(f'(r)\big)^2f''(r)}
  = \lim_{u \rightarrow \infty} \frac{f'''\big(\phi(\sqrt{u})\big)}{3c(u)(1+u)^{3/2}u}.
\]
Therefore $f'''\big(\phi(\sqrt{u})\big) \asymp u^{5/2}$ when $u \rightarrow \infty$, and $F''(u) \asymp u^{-5/2}$.
\end{proof}

With Lemma~\ref{lem:F2} we can estimate higher derivatives of $h$.

\begin{proposition}\label{prop:h1}
Let $(n, m) \in \big(\mathbb{R}^+\big)^2$ with $m \asymp M$. If $n/m^2 < u_{j_0}$ let $U$ be distance of $n/m^2$ to the closest $u_i$, say $u_j$. If $n/m^2 \geq u_{j_0}$ take $U = n/m^2$ and $j = \infty$.Then
\begin{equation*}
  \frac{\partial^3 h}{\partial n^2 \partial m}(n, m) \asymp \frac{U^{d_{j}}}{M^4}. 
\end{equation*}
\end{proposition}

\begin{proof}
By \eqref{F} the partial derivative is $m^{-4}F''(n/m^2)$ and the result follows from Lemma~\ref{lem:F2}.
\end{proof}

\begin{proposition}\label{prop:h2}
Let $(n, m) \in \big(\mathbb{R}^+\big)^2$ with $m \asymp M$ and fix $j$ with $d_j > 0$. If $U = |n/m^2 - u_{j}|$ is small enough, then
\begin{equation*}
 \frac{\partial^3 h}{\partial n \partial m^2}(n, m) \asymp \frac{1}{M^3}. \label{h3}
\end{equation*}
\end{proposition}

\begin{proof}
The partial derivative here is $-2m^{-3}\big(F''(n/m^2)n/m^2+F'(n/m^2)\big)$. By Lemma~\ref{lem:F2} the function $F'$ remains positive and bounded in bounded subintervals of $\mathbb{R}^+$, while $F''(n/m^2)n/m^2\to 0$ when $U\to 0$.
\end{proof}

\begin{proposition}\label{prop:h3}
Let $(n, m) \in \big(\mathbb{R}^+\big)^2$ with $m \asymp M$ and fix $j$ with $d_j > 0$. If $U = |n/m^2 - u_{j}|$ is small enough and $1 \leq \ell \leq UM$,
\begin{equation*}
 \frac{\partial h}{\partial n}(n, m+\ell) - \frac{\partial h}{\partial n}(n, m) = C_{j}\frac{\ell}{m(m+\ell)} + O\left(\frac{\ell U^{d_{j}+1}}{M^2}\right)
\end{equation*}
for some constant $C_{j} \neq 0$.
\end{proposition}

\begin{proof}
We express the left hand side as
\begin{align*}
  \int_0^\ell \frac{\partial^2 h}{\partial n \partial m}(n, m + t) \ dt
  &= \int_0^\ell 
    F'\left(\frac{n}{(m+t)^2}\right) \frac{dt}{(m+t)^2} \\
  &= \int_0^\ell 
    \left[ 
      F'(u_{j}) 
      + \int_{u_{j}}^{n/(m+t)^2} 
      F''(v) \ dv 
    \right] \frac{dt}{(m+t)^2}  \\
  &= F'(u_{j}) \frac{\ell}{m(m+\ell)} + O\left(\frac{\ell U^{d_j+1}}{M^2}\right).
\end{align*}
To bound the error term we have applied Lemma~\ref{lem:F2} noting that $n/(m+t)^2-u_j=O(U)$ for $0\le t\le UM$.
\end{proof}

\section{The van der Corput estimate}\label{s:Ubig}

In this section we generalize the argument in \cite{chamizo} to prove Theorem~\ref{thm:s1} in some ranges and Theorem~\ref{thm:s2} using the simplest van der Corput bound for $T$. The first named author wants to take the opportunity to point out that the case $L=M$ was neglected there (although it does not affect the result). To simplify the proofs, we will assume from now on that $UM \geq R^{3/8}$, as otherwise the trivial estimate $S \ll R^{\epsilon}UM^3$ suffices to prove the desired inequalities. We will also refer to the arguments of $S$ in the statements of these theorems as $U_1$ and $U_2$ for the sake of convenience. 

\begin{proposition}\label{prop:vdC}
Let $R$, $M$, $U$, $U_1$ and $U_2$ be as in the hypothesis of either Theorem~\ref{thm:s1} or \ref{thm:s2}, setting $j = \infty$ in the second case. Then
\begin{multline*}
  \sum_{n}
    e\big(R(h(n, m + \ell) - h(n, m))\big)
  \\
  \ll R^{1/2}\ell^{1/2}U^{(d_j+2)/2} + R^{-1/2}\ell^{-1/2}U^{-d_j/2}M^2,
\end{multline*}
where the range of the summation is $U_1(m+\ell)^2 \leq n < U_2m^2$.
\end{proposition}

\begin{proof}
By the mean value theorem and Proposition~\ref{prop:h1} we have
\[
  \frac{\partial^2}{\partial n^2} \big(h(n, m + \ell) - h(n, m)\big)
  = \ell \frac{\partial^3 h}{\partial n^2 \partial m}(n, \tilde{m})
  \asymp \ell \frac{U^{d_j}}{M^4}.
\]
Applying now the simplest van der Corput well-known estimate (see, for instance, Theorem~2.2 of \cite{grahamkolesnik}),
\[
  \sum_{n}
    e\big(R(h(n, m + \ell) - h(n, m))\big)
  \ll
  UM^2\big(R \ell U^{d_j} M^{-4} \big)^{1/2} + \big(R \ell U^{d_j} M^{-4} \big)^{-1/2}.
\]
This concludes the proof.
\end{proof}

\begin{proposition}\label{2p1}
Theorem~\ref{thm:s1} holds when $d_j = 0$, $1$, or when $d_j \geq 2$ and $U \gg R^{-5/(8d_j-8)}$.
\end{proposition}

\begin{proof}
Note that since $U_2 \ll 1$ we are in position to apply Proposition~\ref{prop:Weyl} as long as we take $L \leq UM$. Using Proposition~\ref{prop:vdC} to bound $T(U_1, U_2, M, L)$ we obtain
\begin{equation}\label{eq:vdC}
  R^{-\epsilon} M^{-4}|S|^2 \ll L^{-1}U^2M^2 + R^{1/2}L^{1/2}U^{(d_j+4)/2} + R^{-1/2}L^{-1/2}U^{(2-d_j)/2}M^2.
\end{equation}
We choose $L = \min\{R^{1/2}U^{-d_j}, UM\}$. If $L = R^{1/2}U^{-d_j}$ then using $M \leq R^{5/8}$ we obtain $M^{-4}|S|^2 \ll R^{3/4+\epsilon}$, as desired. Hence assume $L = UM$ and $U^{d_j + 1} < R^{1/2}M^{-1}$. We have
\[
  R^{-\epsilon} M^{-4}|S|^2 \ll UM + R^{1/2}U^{(d_j+5)/2}M^{1/2} + R^{-1/2}U^{(1-d_j)/2}M^{3/4}.
\]
Using the inequality $U^{d_j + 1} < R^{1/2}M^{-1}$ on the second summand and the hypothesis of this proposition we conclude again $M^{-4}|S|^2 \ll R^{3/4+\epsilon}$.
\end{proof}

\begin{proof}[Proof of Theorem~\ref{thm:s2}]
We proceed similarly as in the previous proof. Note that now $U_2 \asymp U$ and we may take $1 \leq L \leq M$ in Proposition~\ref{prop:Weyl}. Using Proposition~\ref{prop:vdC} to bound $T(U_1, U_2, M, L)$ we obtain exactly the same bound \eqref{eq:vdC} with $d_\infty = -5/2$:
\[
  R^{-\epsilon} M^{-4}|S|^2 \ll L^{-1}U^2M^2 + R^{1/2}L^{1/2}U^{3/4} + R^{-1/2}L^{-1/2}U^{9/4}M^2.
\]
The choice $L = \min\{R^{1/2}, M\}$ also works in exactly the same way, using $U \leq R^{5/4}M^{-2}$ and $M \leq R^{5/8}$, to show $M^{-4}|S|^2 \ll U^{2}R^{3/4+\epsilon}$.
\end{proof}

\section{Diophantine approximation of the phase}

As $U$ gets smaller than $R^{-5/(8d_j-8)}$ the van der Corput estimate is not good enough to prove Theorem~\ref{thm:s1} anymore. The reason is that the phase of the exponential sum in \eqref{eq:T} is almost linear in $n$, as Proposition~\ref{prop:h3} shows, and the oscillation is not captured by a second derivative test.

Throughout this section we will assume that $R$, $M$, $U$, $U_1$, $U_2$ and $j$ are as in the statement of Theorem~\ref{thm:s1}, $UM \geq R^{5/8}$ (see comments in \S\ref{s:Ubig}) and $M \leq m < 2M$. Let $I_{m, \ell} = [U_1(m + \ell)^2, U_2m^2]$, which we may assume non-empty by restricting the possible values of $m$, and define the quantities
\begin{align*}
  \phi_\ell(n, m) &= R\left(\frac{\partial h}{\partial n}(n, m + \ell) - \frac{\partial h}{\partial n}(n, m)\right), \\
  \Phi_\ell(m) &= \min_{x \in I_{m, \ell}} \dist\big(\phi_\ell(x, m), \mathbb{Z}\big).
\end{align*}
The function $\phi_\ell$ is the derivative of the phase of the exponential sum in $n$ appearing in \eqref{eq:T}. Since by Proposition~\ref{prop:h1} it is monotone in $n$, we can apply Kuzmin-Landau's lemma (Theorem~2.1 of \cite{grahamkolesnik}) to obtain the bound
\begin{equation}
  \Bigg| \sum_{n \in I_{m, \ell}}
    e\big(R(h(n, m + \ell) - h(n, m))\big)
  \Bigg|
  \ll \big(\Phi_\ell(m)\big)^{-1}.
\end{equation}
Suppose we can find another bound $H_\ell$ for the same exponential sum, this time uniform in $m$, to apply in those cases when $\Phi_\ell \approx 0$. Then knowing very little about the distribution of the values $\Phi_\ell(m)$ we can find a good bound for $T$. The underlying idea here is to gain from some control of the spacing. In \cite{bombieriiwaniec} and \cite{huxley} this is accomplished via large sieve inequalities, while we introduce the spacing through the following simple result:

\begin{lemma}\label{lem:magic}
Assume we have a finite sequence of points $a_m \in [0, \frac 12]$ satisfying the following condition:
\[
  \# \{ m \, : \, a_m \leq x \}
  \leq A + Bx 
  \qquad \text{for every}\quad 0 \leq x \leq 1/2.
\]
Then for any $H > 0$ we have
\[
  \sum_{m} \min\{H, a_m^{-1}\}
  \leq AH + B(1 + |\log H|).
\]
\end{lemma}
Assuming, in our setting, that $A_\ell$, $B_\ell$ and $H_\ell$ are functions such that $\ell A_\ell$, $\ell B_\ell$ and $\ell H_\ell$ are non-decreasing in $\ell$, and $H_\ell$ is bounded above and below by powers of $R$, it is immediate that for any fixed $\epsilon > 0$,
\begin{equation}\label{eq:Tbound}
  T(U_1, U_2, M, L) \ll R^\epsilon \big(A_L H_L + B_L\big).
\end{equation}

\begin{proof}[Proof of Lemma~\ref{lem:magic}]
Let us say that the finite sequence is $0 \leq a_1 \leq a_2 \leq \cdots \leq a_N \leq 1/2$. Note that, by hypothesis, $m \leq A + Ba_m$. Let $f: [0, \frac 12] \rightarrow \mathbb{R}$ be a non-increasing function and extend it to the negative real numbers as the constant function $f(0)$. Then
\[
  \sum_{m} f(a_m) 
  \leq \sum_{m} f\left( \frac{m-A}{B} \right) 
  \leq B\int_{-A/B}^{1/2}f(x) \, dx 
  = Af(0) + B\int_0^{1/2}f(x) \, dx.
\]
The result follows applying this inequality with $f(x) = \min\{H, x^{-1}\}$.
\end{proof}

The upper bound $H_\ell$ will be either the trivial estimate $UM^2$, or the second term in the van der Corput estimate given by Proposition~\ref{prop:vdC} (the first one may be neglected in the range $U \ll R^{-5/(8d_j-8)}$, $UM \geq R^{3/8}$). The pair $(A_\ell, B_\ell)$ will be given by one of the following two propositions.

\begin{proposition}\label{prop:AB1}
Assume $U^{d_j+1}M$ is small enough and $1 \leq \ell \leq UM$. Then
\[
  \#\{ m \, : \, \Phi_\ell(m) \leq x\}
  \ll 1 + \frac{R \ell}{M^2} + M\left(1 + \frac{M^2}{R\ell} \right)x 
  \qquad \big(0 \leq x \leq 1/2\big).
\]
\end{proposition}

\begin{proof}
Choose for each pair $(m, \ell)$ a point $x_m \in I_{m, \ell}$ (depending implicitly in $\ell$) satisfying
\[
  \Phi_\ell(m) = \dist\big(\phi_\ell(x_m, m), \mathbb{Z}\big).
\]
By the mean value theorem, $\phi_\ell(x_{m+1}, m+1) - \phi_\ell(x_m, m)$ equals
\[ 
  R \ell \frac{ \partial^3 h }{\partial n \partial m^2}(x_1, y_1) 
  + R\ell(x_{m+1} - x_m) \frac{ \partial^3 h }{\partial n^2 \partial m}(x_2, y_2),
\]
for some points $(x_1, y_1)$, $(x_2, y_2)$ lying in the rectangle
\[
  [U_1(m + \ell)^2, U_2(m+1)^2] \times [m, m + \ell + 1].
\]
The function $x/y^2$ over this rectangle satisfies 
\[
  U_1(1-{4}M^{-1}) \leq x/y^2 \leq U_2(1+4M^{-1}),
\]
and since $UM \geq R^{3/8}$ we have $|u_j - x_i/y_i^2| \asymp U$ for $i= 1, 2$. Using the estimates given by Propositions~\ref{prop:h1} and \ref{prop:h2},
\begin{equation}\label{eq:phi}
  \phi_\ell(x_{m+1}, m+1) - \phi_\ell(x_m, m) 
  \asymp \frac{R \ell}{M^3} 
  + O \left( 
    R\ell \cdot UM^2 \cdot \frac{U^{d_j}}{M^4}
  \right)
  \asymp \frac{R \ell}{M^3},
\end{equation}
the sign of the left hand side being always the same.

Since $M \leq m < 2M$, we deduce from \eqref{eq:phi} that the number of integers $k$ satisfying $|\phi_\ell(x_m, m) - k| \leq 1/2$ for some $m$ is at most a constant times $1 + R\ell M^{-2}$. On the other hand we deduce again from \eqref{eq:phi} that for each of those $k$ and any $x\ge 0$
\[
  \# \{ m \, : \,
    |\phi_\ell(x_m, m) - k| \leq x
  \}
  \ll 1 + R^{-1} \ell^{-1} M^3 x.
\]
Therefore,
\[
  \#\{ m : \Phi_\ell(m) \leq x \} 
  \ll \left(1+\frac{R\ell}{M^2}\right) \left(1 + \frac{M^3}{R\ell}x\right)
\]
for every $0\le x\le 1/2$.
\end{proof}

\begin{proposition}\label{prop:AB2}
Fix $\epsilon > 0$. For any $0 < U \leq 1$ we have
\[
  \#\{ m \, : \, \Phi_\ell(m) \leq x\}
  \ll R^\epsilon \big(
    1 + R\ell U^{d_j + 1} + M^2 x
  \big)
  \qquad \big(0 \leq x \leq 1/2\big).
\]
\end{proposition}

\begin{proof}
Let $C_j$ the constant involved in Proposition~\ref{prop:h3}, and assume that we have
\[
  \dist\left(
    C_j \frac{R\ell}{m (m+\ell)}
    , \mathbb{Z}
    \right) \leq x
    \qquad \text{for some}\quad x\ge 0.
\]
This means that there exists an integer $k = k(m, \ell)$ satisfying
\[
  |C_j R\ell - km(m+\ell)| 
  \leq m(m+\ell)x 
  \leq 2M^2 x.
\]
In particular, $m$ must divide a certain integer $km(m+\ell)$ lying in the interval centered at $C_jR\ell$ of half-length $2M^2 x$. Since there are at most $1 + 4M^2x$ of these integers, and each has at most $O(R^\epsilon)$ divisors, we conclude
\[
  \#\big\{m : \dist\big(C_jR\ell/(m(m+\ell)), \mathbb{Z}\big) \leq x\big\} 
  \ll R^\epsilon \big(1+M^2x\big).
\]
Replacing $x$ by $x + O\big(R\ell U^{d_j + 1}M^{-2}\big)$ the result follows from Proposition~\ref{prop:h3}.
\end{proof}

The following two propositions, together with Proposition~\ref{2p1} in \S\ref{s:Ubig}, complete the proof of Theorem~\ref{thm:s1}, and hence also the proof of Theorem~\ref{thethm}.

\begin{proposition}\label{2p2}
If $U \ll R^{-5/(8d_j + 8)}$ for a sufficiently small constant then Theorem~\ref{thm:s1} holds when $d_j \leq 4$, or when $d_j \geq 5$ and $U \gg R^{-5/(4d_j - 16)}$.
\end{proposition}

\begin{proof}
We apply Proposition~\ref{prop:Weyl} to bound $S$ with $L = R^{-3/4}U^2M^2$, which always lies in the interval $[1, UM]$. Using \eqref{eq:Tbound} with $(A_L, B_L)$ given by Proposition~\ref{prop:AB1} (note the hypothesis imply $U^{d_j + 1}M$ is small enough) we obtain
\begin{equation}\label{eq:AB1}
  R^{-\epsilon} M^{-4}|S|^2 \ll
    R^{3/4}
    + \frac{U H_L}{M} \left(1 + \frac{RL}{M^2}\right)
    + U\left(1 + \frac{M^2}{RL}\right).
\end{equation}
We choose either $H_L = UM^2$ or $H_L = R^{-1/8}U^{-(d_j+2)/2}M$ (second term in Proposition~\ref{prop:vdC}) depending on whether $RL/M^2 \leq 1$ or not. In the first case, the right hand side of \eqref{eq:AB1} may be majored by $R^{3/4} + U^2M + R^{-1/4}U^{-1}$, and using $M \leq R^{5/8}$ and $U \geq R^{-1/4}$ (from $UM \geq R^{3/8}$) we conclude $M^{-4}|S|^2 \ll R^{3/4 + \epsilon}$. In the second case, the right hand side of \eqref{eq:AB1} may be majored by $R^{3/4} + R^{1/8}U^{-(d_j - 4)/2} + U$, which also leads to $M^{-4}|S|^2 \ll R^{3/4 + \epsilon}$ under the hypothesis of this proposition.
\end{proof}

\begin{proposition}\label{2p3}
Theorem~\ref{thm:s1} holds when $U \ll R^{-5/(4d_j + 24)}$.
\end{proposition}

\begin{proof}
We proceed similarly as in the proof of Proposition~\ref{2p2}. We apply Proposition~\ref{prop:Weyl} to bound $S$ with $L = R^{-3/4}U^2M^2$, and use \eqref{eq:Tbound} with $(A_L, B_L)$ given by Proposition~\ref{prop:AB2} to obtain
\begin{equation}\label{eq:AB2}
  R^{-\epsilon} M^{-4}|S|^2 \ll
    R^{3/4}
    + \frac{U H_L}{M} \big(1 + RLU^{d_j + 1}\big)
    + UM.
\end{equation}
We choose either $H_L = UM^2$ or $H_L = R^{-1/8}U^{-(d_j+2)/2}M$ depending on whether $RLU^{d_j + 1} \leq 1$ or not. In the first case \eqref{eq:AB2} shows that $M^{-4}|S|^2 \ll R^{3/4 + \epsilon}$ is satisfied trivially, while in the second case the right hand side of \eqref{eq:AB2} may be majored by $R^{3/4} + R^{1/8}U^{(d_j + 6)/2}M^2 + UM$, which also leads to $M^{-4}|S|^2 \ll R^{3/4 + \epsilon}$ under the hypothesis of this proposition.
\end{proof}

\section*{Acknowledgements}

The first author is partially supported by the grant MTM2014-56350-P and the MINECO Centro de Excelencia Severo Ochoa Program SEV-2015-0554. The second author is supported by a ''la Caixa''-Severo Ochoa international PhD programme fellowship at the Instituto de Ciencias Matemáticas (CSIC-UAM-UC3M-UCM).


\end{document}